\documentclass{scrartcl}

\usepackage[dvipsnames]{xcolor}
\usepackage[utf8]{inputenc}  
\usepackage[T1]{fontenc}
\usepackage{kantlipsum}
\usepackage{amsfonts}
\usepackage{tikz}
\usepackage{amsmath,amsthm,amssymb,upref}
\usepackage{stmaryrd}
\usepackage{todonotes}
\usepackage[breakable, theorems, skins]{tcolorbox}
\usepackage[framemethod=TikZ]{mdframed}

\usepackage[natbib,
maxbibnames=4
]{biblatex}
\usepackage{hyperref}
\usepackage{enumerate}
\definecolor{bluebg}{rgb}{0.8,0.8,0.97}
\definecolor{graybg}{rgb}{0.9,0.9,0.9}
\tcbset{enhanced}

\newtheorem{lemma}{Lemma}
\newtheorem{theorem}{Theorem}
\newtheorem{proposition}{Proposition}
\newtheorem{corollary}{Corollary}
\newtheorem{assumption}{Assumption}
\theoremstyle{definition}
\newtheorem{definition}{Definition}
\theoremstyle{remark}
\newtheorem{remark}{Remark}
\newmdtheoremenv [backgroundcolor=bluebg, %
innertopmargin =0pt , %
splittopskip = \topskip, %
skipbelow= 6pt, %
skipabove=6pt, %
topline=false,bottomline=false,leftline=false,rightline=false, roundcorner=10pt]{example}{Example}
\newmdtheoremenv [backgroundcolor=graybg, %
innertopmargin =0pt , %
splittopskip = \topskip, %
skipbelow= 6pt, %
skipabove=6pt, %
topline=false,bottomline=false,leftline=false,rightline=false, roundcorner=10pt]{background}{Background Info}

\newcommand{\R}{\mathbb R}
\newcommand{\A}{\mathcal{A}}
\newcommand{\B}{\mathcal B}
\renewcommand{\H}{\mathcal{H}}
\renewcommand{\P}{\mathbb P}

\newcommand{\calE}{\mathcal E}
\newcommand{\calA}{\mathcal A}
\newcommand{\eps}{\varepsilon}
\renewcommand{\d}{\,\mathrm{d}}

\DeclareMathSymbol{\mlq}{\mathord}{operators}{``}
\DeclareMathSymbol{\mrq}{\mathord}{operators}{`'}

\newcommand{\st}{\,:\,}

\newcommand{\dist}{\operatorname{dist}}
\newcommand{\sign}{\operatorname{sign}}

\title{The lion in the attic -- A resolution of the Borel--Kolmogorov paradox}
\author{
Leon Bungert
\thanks{Hausdorff Center for Mathematics, University of Bonn, Endenicher Allee 62, Villa Maria,
53115 Bonn, Germany. Email: \href{mailto:leon.bungert@hcm.uni-bonn.de}{leon.bungert@hcm.uni-bonn.de}} 
\and 
Philipp Wacker
\thanks{Freie Universität Berlin, Institute for Mathematics, Arnimallee 3
14195 Berlin, Germany. Email: \href{mailto:p.wacker@fu-berlin.de}{p.wacker@fu-berlin.de}}
}

\begin{document}
\maketitle
\begin{abstract}
    The Borel--Kolmogorov paradox of conditioning with respect to events of prior probability zero has fascinated students and researchers since its discovery more than 100 years ago. Classical conditioning is only valid with respect to events of positive probability. If we ignore this constraint and condition on such sets, for example events of type $\{Y=y\}$ for a continuously distributed random variable $Y$, almost any probability measure can be chosen as the conditional measure on such sets. There have been numerous descriptions and explanations of the paradox' appearance in the setting of conditioning on a subset of probability zero. However, most treatments don't supply explicit instructions on how to avoid it. We propose to close this gap by defining a version of conditional measure which utilizes the Hausdorff measure. This makes the choice canonical in the sense that it only depends on the geometry of the space, thus removing any ambiguity. We describe the set of possible measures arising in the context of the Borel--Kolmogorov paradox and classify those coinciding with the canonical measure. The objective of this manuscript is to provide a manual for singular conditional probability: We give an explicit explanation in which settings ambiguity arises (and where not) and how to get rid of this ambiguity once and for all by a canonical choice.
    

\textbf{MSC classes:} 60A10, 62A01, 62F15  
\end{abstract}

\newpage
\tableofcontents

\section{Introduction}
Imagine moving into a new house. 
The realtor shows you all the rooms with exception of the attic, which he explicitly tells you not to enter because of a lion living inside which will attack anyone entering. 
Imagine accepting this and living in the house under this restriction. 
You realize, though, that you need the attic space so you carefully work out the lion's sleeping schedule and tiptoe in and out whenever you need to retrieve things you have stored in the its den. 
Usually that works out just fine and only rarely do you lose a guest when you have failed to inform them about your safety and avoidance protocols.

The concept of conditional probability and expectation is such a lion in the house of the mathematical and statistical sciences and its teeth (most notably its fang, the Borel--Kolmogorov paradox) have taken countless victims. Almost everyone who has studied conditioning knows that ``you have to be careful'', especially when conditioning on singular events. There are multiple ways in which scientists deal with that fear of conditioning. Some, similar to our houseowner, only use conditioning in a specific way such that ambiguity cannot arise. Others try to screen the difficulties with bigger words like ``regular conditional probability'' or conditional expectation with respect to sigma-algebras (but those are solutions for different problems and don't have a bearing on the paradox). The pure probabilist's take on this issue is that it is not of interest at all because conditioning on singular sets is a fool's errand. Nonetheless, conditional densities which are ubiquitous in practice are (supposedly) attempting to do exactly that. 
In the authors' opinion there is no accessible textbook or paper which completely solves the problem wherefore we are trying to close this gap.

The aim of this article is to provide the necessary tools and an easy set of rules for entering the attic anytime and getting out alive.

We want to note that this manuscript can be neither written nor read from a purely probabilistic viewpoint: Probability theory's judgement on singular conditioning is that it is impossible. 
Since this provision is not heeded by researchers in the fields of Bayesian inversion or applied statistics, we need to find a way to untangle the arising confusion. 
For this we need geometrical tools, thereby leaving the domain of probability theory. 
In particular, in this work we will only consider finite-dimensional probability spaces which carry a natural metric. 
This is undoubtedly a long way from the generality of the usual probability space triple, but one which most practitioners will not find unduly restricting for their purposes.

There have been numerous attempts to warn of and explain the Borel--Kolmogorov paradox which goes back to \cite{bertrand1889calcul,borel1909ements,kolomogoroff2013grundbegriffe}.
The most accessible and helpful one is \cite{proschan1998expect}. It gives a good visual explanation of the reasons for the occurence of the paradox (which we adopt here because it is very revealing) but ultimately does not point towards a solution (thus falling into the ``just be careful'' category). The original Borel--Kolmogorov paradox (with a measure on the sphere) is also demonstrated in the well-known textbook \cite{billingsley2008probability}, but again with no complete solution on how to deal with it. 
The probably most ambitious opus is \cite{rao2005conditional} which lays out the fundamental derivation of conditional probability via conditional expectation, points out all major pitfalls including the Borel--Kolmogorov paradox. It then goes on to elaborate on Renyi's axiomatic approach to conditional expectation but finally concludes that this does not solve the problem either. See also \cite{arnold2019conditional} and \cite{pfanzagl1979conditional}. For visually appealing geometric explanations and philosophical justification of the Borel--Kolmogorov paradox, see \cite{rescorla2015some} and \cite{meehan2020borel}. 

For comments on non-standard ways of introducing conditional distributions, we also refer to~\cite{tjur1975constructive,chang1997conditioning,dellacherie1982probabilities}.

An interesting variation on the Borel--Kolmogorov paradox was discussed in \cite{kac1959large}. The authors show that there are again many plausible ways of approximating a singular set. Their resolution is to use physical arguments in order to single out a specific version.

More recently, \cite{gyenis2017conditioning} attempted to resolve the Borel--Kolmogorov paradox by extending an elementary $0$-$1$-measure defined on the singular domain to the whole probability space. Unfortunately, their construction does not work, see the erratum\footnote{http://phil.elte.hu/gyz/counterexample.pdf} which we detail in the appendix of this paper. 
On a fundamental level, the problem is that their proposed functional $q_{\mathcal A}$ is not absolutely continuous with respect to the prior. 

This manuscript at hand is not addressed at those mathematicians/statisticians who have never had any issues with wrapping their head around the concept of conditional expectation. In contrast, it is supposed to help everyone who has been troubled by conditional probability much too long.

\paragraph{Plan of this paper}
Our article is structured in the following way: for readers yet unaware of the lion's fangs, we start with a brief outline of the Borel--Kolmogorov paradox (which we generalize such that a conditional measure on a set of measure zero can be almost anything). The main problem is that different parametrizations (i.e. choices of auxiliary random variables) of the singular event in question lead to different results. We then introduce the notion of \textit{canonically induced measure} and show how this singles out one of the various possible measures arising from conditioning on a set of measure zero. This removes the ambiguity at the heart of the paradox. This canonically induced measure is hard too work with in practice, so we then characterize the equivalence class of auxiliary random variables which in fact coincide with the canonically induced measure, thereby giving a practical rule for concrete computation.
A concluding section connects the Borel--Kolmogorov paradox with the field of Bayesian inversion (or Bayesian statistics). The paradox does not typically appear here, because in this case we do not condition on a singular set, but rather on a singular event of a specific choice of random variable. This removes the ambiguity at the heart of the Borel--Kolmogorov paradox. Nevertheless, it pays off to exactly understand why this is the case.

\section{Guiding example: The paradox in an elementary statistical problem}\label{sec:guidingexample}
We assume\footnote{This example is from \cite{proschan1998expect} and is similar to an earlier computation in \cite{rao1988paradoxes}} that there are two i.i.d. random variables $X,Y\sim N(0,1)$. We interpret $X$ as an unknown parameter and $Y$ as an unknown noise term. We further assume that there is a dependent random variable $Z$ (called ``data'') defined via 
\[Z = \frac{Y}{X},\]
i.e., the data $Z$ is a measurement of $1/X$ corrupted by multiplicative noise $Y$.

Now assume we are given a concrete measurement result $Z=-1$. What information does this give us about the parameter $X$? Interpreted in a Bayesian way, we should take our prior on $X$, which is $N(0,1)$ and update it with the likelihood\footnote{for now, we will formally write things like $\P(Z=-1)$ in order to describe the setting more clearly. Of course we have to substitute this with the density of $Z$ etc., but this will come later.} $\P(Z=-1\vert X)$ in order to obtain the posterior $\P(X\vert Z=-1)$. Now obviously, the event that we condition on has probability zero so we have to depart from a simple application of Bayes' law of the form 
$$\P(X=x\vert Z=-1) = \frac{\P(Z=-1\vert X=x)\cdot \P(X=x)}{\P(Z=-1)}$$ 
because all quantities involved are singular. Hence, we do what we teach our students and substitute the singular probabilities by densities. Here we will always write $f_W(w)$ for the density of a continuous random variable $W$ on $\R$. If there are two random variables $A,B,C$ where $C$ is defined via $C= F(A, B)$ (as in the case of $X, Y, Z$ here) we write $f_{C\vert A}(c\vert a) = f_{C\vert A=a}(c)$ as the density of the random variable $F(a, B)$ (i.e., $A=a$ is fixed and only $B$ accounts for randomness in the new random variable $C\vert (A=a)$).\footnote{We don't further formalize this approach here because it is such a standard method in any textbook about dependent continuous random variables and because we want to show how this leads to problems.}

In our guiding example we have $f_X(x) = \frac{1}{\sqrt{2\pi}}\exp(-x^2/2)$ and $f_Y$ has the same form. The density of $Z\vert (X=x)$ is given by the density of $Y/x$, which is 
$$f_{Z\vert (X=x)}(z) = \frac{1}{\sqrt{2\pi \frac{1}{x^2}}}\exp(-z^2x^2/2)$$ 
for $x\neq 0$ (and undefined else, but $X=0$ is an event of probability $0$). In order to get the marginal distribution of $Z$, we need to integrate over the density of $X$, which, after some calculation, yields $f_Z(z) = \frac{1}{\pi(1+z^2)}$ (unsurprisingly, because this is the density of a Cauchy distribution, which is the result of taking the fraction of two standard Gaussian random variables).

The version of Bayes' theorem for densities now yields
\[f_{X\vert Z}(x\vert z) = \frac{f_{Z\vert X}(z\vert x)\cdot f_X(x)}{f_Z(z)} = \frac{\frac{\vert x\vert }{\sqrt{2\pi}}\exp(-z^2x^2/2)\cdot \frac{1}{\sqrt{2\pi}}\exp(-x^2/2)}{\frac{1}{\pi (1+z^2)}}\]
and in particular 
\[f_{X\vert (Z=-1)}(x) = f_{X\vert Z}(x\vert -1) = \vert x\vert \cdot \exp(-x^2).\]

\paragraph{A different approach} If this calculation is deemed too tedious, one can alternatively utilize that $Z = -1$ is equivalent to $X + Y = 0$ and define another random variable $W = X+Y$. Then $\{Z = -1\} = \{W = 0\}$ (this is correct) and we can equivalently condition on $(W=0)$ (this is incorrect, and this is where we lose an arm to the lion because we were not cautious enough). Now, calculations are easier: 
$f_{W\vert X}(w\vert x) = \frac{1}{\sqrt{2\pi}}\exp(-(w-x)^2/2)$ and $f_W(w) = \frac{1}{\sqrt{4\pi}}\exp(-w^2/4)$. This even follows without marginalizing and is due to $W = N(0,2)$ as a sum of two i.i.d. standard random variables. Then again,
\[f_{X\vert W}(x\vert w) = \frac{f_{W\vert X}(W\vert x)\cdot f_X(x)}{f_W(w)} = \frac{\frac{1}{\sqrt{2\pi}}\exp(-(w-x)^2/2) \cdot \frac{1}{\sqrt{2\pi}}\exp(-x^2/2)}{\frac{1}{\sqrt{4\pi}}\exp(-w^2/4)}\] 
and
\[f_{X\vert W=0}(x) = f_{X\vert W}(x\vert 0) = \frac{1}{\sqrt{\pi}}\exp(-x^2),\]
i.e., $X\vert (W=0)$ is a $N(0, 1/\sqrt 2)$ random variable.

Now, although clearly $\{Z = -1\} = \{W = 0\}$, the conditional densities of $X$ conditioned to either event are completely different (most notably: the former approach assigns very small posterior probability to intervals centered  at $x=0$ whereas the latter approach identifies $x=0$ as the point of highest probability density), as figure \ref{fig:densities} shows.

\begin{figure}[hbtp]
\centering
\includegraphics[width=\textwidth]{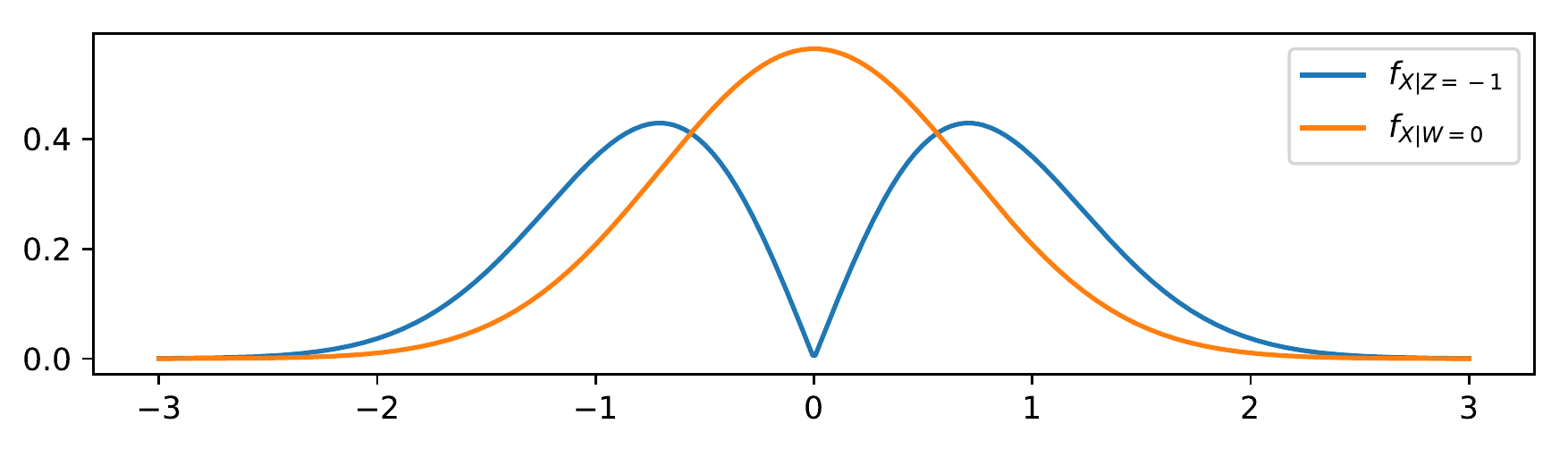}
\caption{Conditional densities $f_{X\vert Z=-1}$ and $f_{X\vert W=0}$.}
\label{fig:densities}
\end{figure}

More pointedly we can say the following: Given $X, Y, W, Z$ as above, the events $\{W = 0\}$ and $\{Z = -1\}$ are equivalent (i.e., they differ only on a set of zero measure\footnote{Note that the fact that $Z$ is not defined for $X=0$ is not relevant for the occurrence of this paradox.}), but our updated belief about $X$ given one of those events, i.e., the distributions $X\vert \{W = 0\}$ and $X\vert \{Z = -1\}$ differ strongly. In particular, the former identifies $X=0$ as the point of highest probability density. The latter says the opposite.

We can produce a countless number of different results by, e.g., defining $R := (X+Y)\cdot \exp((X-Y)^2)$ and conditioning on $R=0$.

We call this the (extended) Borel--Kolmogorov paradox: The original version of this paradox was formulated for a uniform measure on the sphere with a specific parametrization via spherical coordinates. There, the conditional measure on the equator is different from the conditional measure on the ``Greenwich meridian'', which is paradoxical as both sets are equivalent to each other via rotation. For a discussion see for example \cite{billingsley2008probability}. But as the manifold setting needlessly complicates the mathematics -- obfuscating the real issue -- and the same problems lead to the same paradox in flat space, we will stick with Euclidean setting here, with the generalization to manifolds being a cumbersome corollary (not included in this manuscript).

\paragraph{When the Borel--Kolmogorov paradox happens and how we can we avoid it}

In brevity, the Borel--Kolmogorov paradox arises when one wants to condition on a \textit{set of zero measure} but uses \textit{different parametrizations} of this subset via random variables.

It is \textit{not directly relevant} in the case where we are interested in directly conditioning on a singular event of the form $W=w$ for some random variable $W$ and value $w$.

One (pragmatic) approach to avoiding it is to always specify the parametrization, i.e., don't ask for ``what is the distribution of $X$ given $X=Y$?'', but rather ``What is the distribution of $X$ given $W=0$ for $W:= X-Y$?''. This is of course not helpful if the initial task is to ``condition on $M$'' where $M$ is a singular set, and if there is no specific parametrization given. The paradox tells us that different choices of auxiliary level set random variables (i.e. random variables $S$ such that $; = \{S=s\}$) lead to different results. For this situation we need to be able to single out a measure that gives us the ``canonical'' answer.

In this manuscript we argue that there is such a ``canonical'' choice of conditional measure on the subset without the need for additional random variables, if the probability space is Euclidean (and thus carries a metric). We will call this the canonically induced measure and it is purely defined via the metric of the probability space and the underlying probability measure. We will then specify the class of parametrizations of the subset we condition on which recover the canonically induced measure. We also discuss why the paradox is not directly relevant in the Bayesian setting, i.e. if we condition on a specific variable.

Very briefly, the statements of this manuscript can be summarized as follows: Let $\mu$ be a probability measure on $\R^n$ and $M$ a smooth manifold of dimension $k<n$. This manifold arises as the level set $\{\varphi = s\}$ for some smooth function $\varphi :M\to \R$ which can be interpreted as a random variable. Hence, conditioning on $M$ can be interpreted as conditioning on the event $\{\varphi = s\}$, but as there are infinitely many possible mappings $\varphi$ which have this property, there is considerable ambiguity which leads to the Borel--Kolmogorov paradox. We call this type of conditioning \textit{fan measure}.

There is a \textit{canonically induced measure} $\mu_M$ on $M$ directly resulting from the metric structure of $\R^n$ via the Hausdorff measure and hence completely independent of the parametrization of $M$ by arbitrary mappings. This choice of submeasure avoids the Borel--Kolmogorov paradox.

As the canonically induced measure $\mu_M$ is difficult to compute, we derive conditions for $\varphi$ such that the respective fan measure coincides with the canonically induced measure. The crucial condition is
\[ \vert D \varphi\vert  \equiv \text{constant on } M \]
which can be geometrically interpreted as the sets $\{\varphi \in (s-\eps, s+\eps)\}$ being tubes of constant width around $M$ for $\eps>0$ small enough.

\section{Defining measures on a subset of lower dimension}\label{sec:resolutionBK}
Our setting will entail a probability space which is a subset of $\R^n$ and a singular event $M$ on which we would like to condition some probability measure $\mu$.
For all purposes here, this singular event will be a manifold of lower dimension than the probability space $\R^n$. If the measure does not have any atoms on this manifold, then it indeed constitutes a singular event, i.e. $\mu(M) = 0$.

We will define two notions of measures on such a submanifold $M$. The first notion we will call the ``canonically induced measure''. Its definition will be completely independent of any charts that are used in the definition of the manifold. We argue that this is the ``correct'' (or at least a canonical) way of defining a measure which is induced on a lower dimensional structure embedded in the probability space.
The second notion is a chart-dependent definition. As a given manifold can be described by various charts, there is not a unique way of doing that. We will show when and how those two notions coincide and we will argue that the Borel--Kolmogorov paradox happens when we compare instances of the second notion with different charts. Our proposed resolution of the paradox is to always use the version generated by the canonically induced measure, \textit{unless} one is explicitly interested in a specific parametrization given by a random variable, see section~\ref{sec:solution}.

\subsection{Canonically induced measure}

\begin{definition}[Canonically induced measure] \label{def:canonical_submeasure}
Let $\mu$ be an absolutely continuous probability measure on $(\R^n,\mathcal B(\R^n))$ with lower semicontinuous non-negative density $f_\mu$, and $M\subset\R^n$ be a $r$-rectifiable set with $\H^r(M)>0$ (this is in particular true if $M$ is a smooth submanifold of dimension $r$), where $0 \leq r \leq n$. Then we define a functional on subsets $A\subset M$ by setting
\begin{align}\label{eq:canonical_submeasure}
    \mu_{M}(A) := \frac{\int_A f_\mu \d\H^{r}}{\int_M f_\mu \d\H^{r}},
\end{align}
where $\H^r$ denotes the $r$-dimensional Hausdorff measure. We call $\mu_M$ the \textit{canonically induced measure} on $M$, and ``canonical'', because it is essentially defined via the natural metric given by the Euclidean distance. 
\end{definition}

\begin{remark}
Obviously, the measure $\mu_M$ is a generalization of the measure $\mu$ since by choosing $r=n$ it holds that $\mu_{\R^n}=\mu$.
\end{remark}

\begin{remark}
Note that the definition of $\mu_M$ mirrors in a very natural way the definition of conditional probability in the non-singular case, i.e., $\mu(D_1\vert D_2) := \frac{\mu(D_1\cap D_2)}{\mu(D_2)}$ for a probability measure $\mu$ and $\mu(D_2) > 0$. The Hausdorff measure allows us to bypass the issue of ``$0/0$''.

For this reason, we will also interpret $\mu_M$ as the ``canonical'' conditional probability of $A$ given $M$, or $\mu(A\vert M) := \mu_M(A).$
\end{remark}

\begin{figure}[hbtp]
\centering
\includegraphics[width=0.9\textwidth]{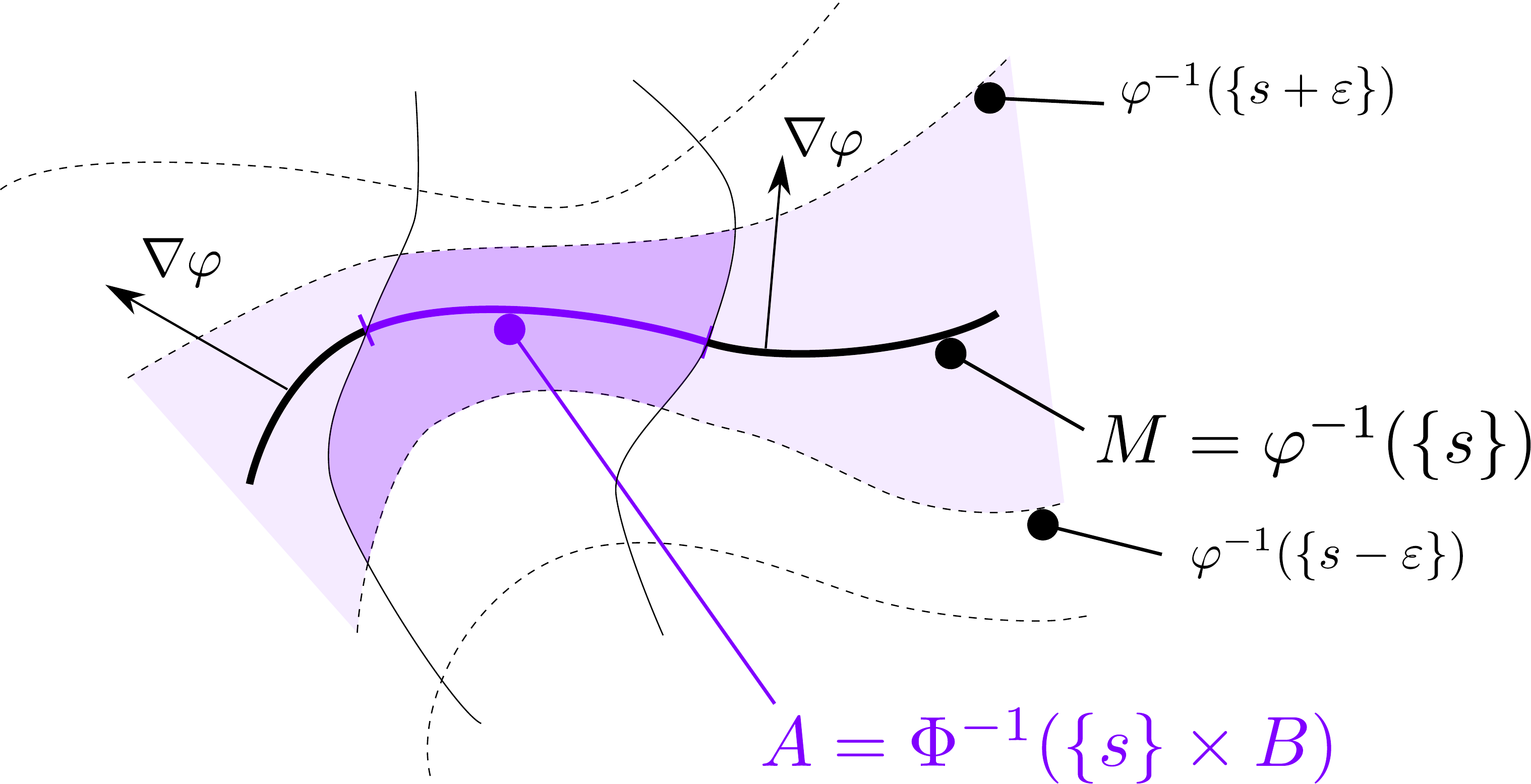}
\caption{Visualization of the fan measure: $\mu_{\Phi, s}(A)$ is the limit superior of the ratio of the areas enclosed between ``fans'' of values $s\pm \eps$.}
\end{figure}
The measure $\mu_M$ inherits being a metric outer measure from the Hausdorff measure.
\begin{lemma}
$\mu_M : \mathcal P(M) \to [0, 1]$ is a metric outer measure. 
\end{lemma}
Next, we define a chart-dependent measure which usually differs from $\mu_M$ but is used implicitly every time we condition on a random variable. This is a favourable approach in practice because it can be computed in practice. This ease of application comes with the ambiguity at the heart of the Borel--Kolmogorov paradox.

\subsection{Fan measure}
Let $(\R^n, \mathcal B(\R^n), \mu)$ be a probability space and $X$ a random variable on some abstract probability space $(\Omega,\A,\P)$ with values in $\R^n$ such that $\mu$ is the law of $X$, meaning
\begin{align*}
    \P(X\in C) = \mu(C).
\end{align*}
We start by defining a way of interpreting conditioning as ``looking at a submanifold'': to this end we consider the function 
\begin{align*}
    \Phi : \R^n \to \R^m,\quad
    x \mapsto 
    \begin{pmatrix}
        \varphi(x)\\\psi(x)
    \end{pmatrix},
\end{align*}
where $\varphi : \R^n \to \R^k$ and $\psi: \R^n \to \R^{m-k}$, $k\in\{1,2,\dots,m-1\}$, are continuous functions. 
We define a family of $n-k$-dimensional submanifolds of $\R^n$ by looking at the level sets of $\varphi$:
\[M^s := \{ x\in \R^n: \varphi(x) = s\},\]
where $s\in \R^k$, with $k=1$ being the application that we usually have in mind.

We assume that $\mu(\bigcup_{s\in S} M^s) = 1$ for some parameter set $S$. The family can be thought of as fanning out the whole space (up to sets of measure $0$). 

\begin{definition}[Fan measure]\label{def:fan_measure}
Let $\mu$ be a \textit{strictly positive} measure on $\R^n$, meaning that $\mu(A)>0$ for every open set $A\subset\R^n$.
Furthermore, we define $M^s = \Phi^{-1}(\{s\}\times \R^{m-k})$ and the associated trace $\sigma$-algebra
\begin{align}\label{eq:trace_sigma-alg}
    \A^s:={M^s}\cap\psi^{-1}(\mathcal B(\R^{m-k}))=\Phi^{-1}(\{s\}\times \mathcal B(\R^{m-k})).    
\end{align}
Then for $A=\Phi^{-1}(\{s\}\times B)\in\A^s$ we define
\begin{align}\label{eq:fan_submeasure}
    \mu_{\Phi, s}(A) := \limsup_{\eps\searrow 0}\frac{\Phi_\#\mu(B_\eps(s)\times B)}{\Phi_\#\mu(B_\eps(s)\times \R^{m-k})},
\end{align}
where $\Phi_\#\mu:=\mu\circ\Phi^{-1}$ is the push-forward measure of $\mu$ under $\Phi$ and $B_\eps(s):=\{x\in\R^k\st \vert x-s\vert <\eps\}$ denotes the open ball around $s\in\R^k$ with radius $\eps>0$. We call $\mu_{\Phi, s}$ the \textit{fan measure} associated to $\Phi$ (because the mapping $\Phi$ fans out the whole space and each submanifold is one fold of the fan).
\end{definition}

We will now see that the fan measure exactly corresponds to the usual notion of conditional distribution (under relatively mild assumptions on the random variables involved).
Consider an abstract probability space $(\Omega, \mathcal A, \P)$, a random variable $V: \Omega \to \R$ with density $f_V$ with respect to the Lebesgue measure, and a random vector $U: \Omega \to \R^{m-1}$ with joint density $f_{U,V}:\R^m\to \R$. 
We assume that the densities are continuous and that $f_Z$ is strictly positive.

For $a<b$, we can without ambiguity calculate the probability 
$$\P(U\in B\vert  V\in (a,b)) = \left({\int_a^b \int_B f_{U,V}(u,v) \d u\d v}\right){\bigg/}\left({\int_a^b f_V(v)\d v}\right).$$

To condition on $V=s$ one usually employs the singular limit (if it exists)
\begin{align} \label{eq:singular_cond_prob}
\P(U\in B\vert V=s) :=  \limsup_{\eps \searrow 0} \P(U\in B\vert V \in (s-\eps, s+ \eps)).
\end{align}
If one uses the densities $f_{U,V}$ and $f_V$ one can use this definition to calculate the conditional density using the Lebesgue-Besicovich differentiation theorem (or the fundamental theorem of calculus):
\begin{align*} 
\P(U\in B\vert V=s) &=  \limsup_{\eps \searrow 0} \P(U\in B\vert V \in (s-\eps, s+ \eps)) \\
&=\limsup_{\eps\searrow 0} \frac{\int_{s-\eps}^{s+\eps}\int_B f_{U,V}(u,v)\d u \d v}{\int_{s-\eps}^{s+\eps}f_V(v)\d v}\\
&= \lim_{\eps\searrow 0} \frac{\frac{1}{2\eps}\int_{s-\eps}^{s+\eps}\int_B f_{U,V}(u,v)\d u\d v}{\frac{1}{2\eps}\int_{s-\eps}^{s+\eps}f_V(v)\d v}\\
&= \frac{\int_B f_{U,V}(u,s) \d u}{f_V(s)}
\end{align*}
Thus, the conditional density of $U$ given $V=s$ for some $s\in \R$ is defined as 
\begin{align}\label{eq:cond_density}
    f_{U\vert V=s}(u) = \frac{f_{U,V}(u,v)}{f_V(s)}.
\end{align}
This is the way one usually defines conditional densities (see for example section 1.4 in \cite{rao2005conditional}).

Now we can see that the fan measure corresponding to fans of co-dimension one is directly related to conditioning a random variable on a subset of measure zero:

\begin{lemma}\label{lem:conditioning_on_rv}
Let $X:(\Omega, \mathcal A, \P)\to (\R^n, \mathcal B(\R^n))$ be a random variable with law $\mu$, i.e., $\P(X\in C) = \mu(C)$. Define the random variable $Z := \varphi(X)$ and the random vector $W := \psi(X)$ (i.e., $(Z, W) = \Phi(X)$) for continuous mappings $\varphi: \R^n\to \R$ and $\psi: \R^n\to \R^{m-1}$, and write $\Phi = (\varphi,\psi)$. 

Then, for $A:=\Phi^{-1}(\{s\}\times B) \in \A^s$ it holds
\[ \mu_{\Phi, s}(A) =  \P(W \in B\vert  Z = s),\]
where $\mu_{\Phi,s}$ is as in Definition~\ref{def:fan_measure} for $k=1$ and $\P(W \in B\vert  Z = s)$ is defined by \eqref{eq:singular_cond_prob}.
\end{lemma}
\begin{proof}
Since $\mu$ is a strictly positive measure and $\varphi$ is continuous, it holds for any open interval $I\subset\R$
$$\P(Z\in I)=\P(\varphi(X)\in I)=\P(X\in\varphi^{-1}(I))=\mu(\varphi^{-1}(I))>0.$$
Hence, we can use the definition of the conditional probability to obtain
\begin{align*}
\P(W\in B\vert Z = s) &= \limsup_{\eps \searrow 0} \P(W\in B\vert Z \in (s-\eps, s+ \eps))\\
&= \limsup_{\eps \searrow 0} \frac{\P(W\in B,~Z \in (s-\eps, s+ \eps))}{\P(Z \in (s-\eps, s+ \eps))}\\
&= \limsup_{\eps \searrow 0} \frac{\Phi_\#\mu((s-\eps,s+\eps)\times B) }{\Phi_\#\mu((s-\eps,s+\eps)\times \R^{m-1})} \\
&=\mu_{\Phi,s}(A),
\end{align*}
where in the last step we utilized that $B_\eps(s)=(s-\eps,s+\eps)$ in dimension $k=1$.
\end{proof}
\begin{corollary}
The fan measure \eqref{eq:fan_submeasure} for $k=1$ is a probability measure if and only if the singular conditional probability \eqref{eq:singular_cond_prob} is. 
This is the case, for instance, if the latter has the conditional density \eqref{eq:cond_density}.
\end{corollary}

More concretely, this means that different fan measures on a manifold $M$ correspond to using different (artificial) random variables to condition on this manifold. This is the reason for the ambiguity in the Borel--Kolmogorov paradox: There is no a-priori canonical random variable which we need to condition on in order to obtain the ``correct'' restriction on $M$. 

On the other hand, the canonically induced measure is defined without additional random variables and in a way that only uses the metric of the space (incorporated through the Hausdorff measure). For this reason, it is interesting to investigate which fan measures yield the same measure as the canonical measure.

\subsection{Conditions for equivalence of canonical and fan measure}

\begin{theorem}[Equality of canonical and fan measure]\label{thm:main}
Let $A = \Phi^{-1}(\{s\}\times B) \in \A^s$ for some $s\in\R^k$ and some set $B\subset\B(\R^{m-k})$. Denote $J\varphi(x):=\sqrt{\det(D\varphi(x)D\varphi(x)^T)}$ for $x\in\Omega$ and assume that
\begin{itemize}
    \item there exists a constant $C>0$ such that $J\varphi \equiv C$ on $M^s$,
    \item there exists a constant $r>0$ and an open neighborhood $U$ of $M^s$ such that the smallest singular value of $D\varphi$ is bounded from below by $r$ on $U$, 
    \item $J\varphi\in C^1$ on $U$.
\end{itemize}
Then it holds for almost all $s\in\R^k$
\begin{align}
    \mu_{\Phi,s}(A) = \mu_{M^s}(A).
\end{align}
\end{theorem}
\begin{proof}
We recall the coarea formula \citep{evans2015measure}
\[\int_{\R^n}g(x) J\alpha(x)\d x = \int_{\R^k}  \int_{\alpha^{-1}({t})} g(x) \d\H^{n-k}(x) \d t   \]
for functions $\alpha:\R^n \to \R^k$ and $g:\R^n\to\R$ where $J\alpha(x) = \sqrt{\det(D\alpha(x)D\alpha(x)^T)}$. Using this formula with $\alpha=\varphi$ and $g=\chi_{\Phi^{-1}(B_\eps(s)\times B)}f_\mu$ we obtain
\begin{align}\label{eq:coarea_measures}
    \int_{\Phi^{-1}(B_\eps(s)\times B)}J\varphi(x)f_\mu(x)\d x = \int_{B_\eps(s)}\int_{\Phi^{-1}(\{t\}\times B)}f_\mu(x)\d\H^{n-1}(x) \d t.
\end{align}
Let us choose $\eps>0$ small enough such that $\Phi^{-1}(B_\eps(s) \times B) \subset U$. Then we can use that $J\varphi$ is continuously differentiable to obtain $J\varphi(x)=C+o(\dist(x,M^s))$ for $x\in \Phi^{-1}(B_\eps(s) \times B)$.
To estimate $\dist(x,M^s)$ we study the gradient flow
\begin{align*}
    \dot{z}(t) = -A(t)\sign(\varphi(z(t))-s),\quad z(0)=x,
\end{align*}
where $A(t)\in\R^{n\times k}$ is a right inverse of $D\varphi(z(t))\in\R^{k\times n}$.
Note that such a right inverse is given by $A(t) := D\varphi(z(t))^T(D\varphi(z(t))D\varphi(z(t))^T)^{-1}$ which is well-defined on $U$ by our assumptions.
We used the notation
\begin{align*}
    \sign:\R^k\to\R^k,\quad
    \sign(x)=
    \begin{cases}
    \frac{x}{\vert x\vert }, \quad &x\neq 0 \\
    0,\quad &x=0.
    \end{cases}
\end{align*}
Along the solution $z(t)$ it holds by the chain rule:
$$\frac{\d}{\d t}\varphi(z(t)) = -\sign(\varphi(z(t))-s).$$ 
Hence, the curve $\eta(t):=\varphi(z(t))-s$ solves the gradient flow of the Euclidean norm in~$\R^k$, given by
$$\frac{\d}{\d t}\eta(t) = -\sign(\eta(t)),\quad \eta(0)=\varphi(x)-s.$$ 
This flow has an explicit solution \citep{bungert2019nonlinear} given by
$$\eta(t) = \max(1-\lambda t,0)\,\eta(0),\quad \lambda=\frac{1}{\vert \eta(0)\vert },$$
which becomes extinct at time $$T:=1/\lambda=\vert \eta(0)\vert =\vert \varphi(x)-s\vert .$$
Hence, it holds $\eta(T)=0$ which is equivalent to $\varphi(z(T))=s$ and $y:=z(T) \in M^s$.
Note furthermore that $T$ meets $0<T<\eps$.
It holds that $\|A(s)\|=\frac{1}{\sigma_{\min}}$, where $\sigma_{\min}$ is the smallest singular value of the matrix $D\varphi(z(s))$ and by assumption is larger or equal than $r>0$ on $U$.
Hence, we obtain
\begin{align*}
\vert x-y\vert  &= \vert z(0) - z(T)\vert  \leq \int_0^T \|A(s)\| \d s \leq \frac{T}{r} < \frac{\eps}{r}.
\end{align*}
This shows that $o(\dist(x,M^s))=o(\eps)$ and hence by using $J\varphi(x)=C+o(\eps)$ in \eqref{eq:coarea_measures} we obtain 
\begin{align*}
    C\cdot \Phi_\#\mu(B_\eps(s) \times B) = \int_{B_\eps(s)}\int_{\Phi^{-1}(\{t\}\times B)}f_\mu(x)\d\H^{n-k}(x) \d t + o(\eps).
\end{align*}
Using this identity also for $B=\R^{m-k}$ and cancelling the factor $C>0$ yields
\begin{align*}
    \mu_{\Phi,s}(A) 
    &= \lim_{\eps\searrow 0} \frac{\Phi_\#\mu(B_\eps(s) \times B)}{\Phi_\#\mu(B_\eps(s) \times \R^{m-k})} \\
    &= \lim_{\eps\searrow 0} \frac{\int_{B_\eps(s)}\int_{\Phi^{-1}(\{t\}\times B)}f_\mu(x)\d\H^{n-k}(x) \d t + o(\eps)}{\int_{B_\eps(s)}\int_{\Phi^{-1}(\{t\}\times \R^{m-k})}f_\mu(x)\d\H^{n-k}(x) \d t + o(\eps)} \\
    &= \lim_{\eps\searrow 0} \frac{\frac{1}{\vert B_\eps(s)\vert }\int_{B_\eps(s)}\int_{\Phi^{-1}(\{t\}\times B)}f_\mu(x)\d\H^{n-k}(x) \d t}{\frac{1}{\vert B_\eps(s)\vert }\int_{B_\eps(s)}\int_{\Phi^{-1}(\{t\}\times \R^{m-k})}f_\mu(x)\d\H^{n-k}(x) \d t}.
\end{align*}
Now we observe that $\Phi^{-1}(\{t\}\times B)=\varphi^{-1}\{t\}\cap\psi^{-1}B$ is the intersection of a level set and a measurable set.
Hence, thanks to \cite[Lemma 3.5]{evans2015measure}, the function
$$t\mapsto \int_{\Phi^{-1}(\{t\}\times B)}f_\mu(x)\d\H^{n-k}(x)$$
is Lebesgue measurable.
This allows us to use the Lebesgue differentiation theorem to obtain for almost all $s\in\R^k$:
\begin{align*}
   \mu_{\Phi,s}(A) 
    = \frac{\int_{\Phi^{-1}(\{s\}\times B)}f_\mu(x)\d\H^{n-k}(x)}{\int_{\Phi^{-1}(\{s\}\times \R^{m-k})}f_\mu(x)\d\H^{n-k}(x)} 
    = \mu_{M^s}(A).
\end{align*}
\end{proof}

Now we have established a way of dealing with the Borel--Kolmogorov paradox: If we want to condition on a singular event, we don't need to use an arbitrary parametrization of this set as a level set of an additional auxiliary random variable. Instead, we can by default choose the canonically induced submeasure on this event. As this measure is difficult to work with in practice, theorem \ref{thm:main} tells us exactly which (equivalence class of) auxiliary random variables we can use in order to use the power of conditional densities for a concrete computation.

\subsection{The fan measure if the transformation mapping is a diffeomorphism}
In practice, the transformation used often is a diffeomorphism. In the guiding example, the two choices correspond to $\Phi(x) = (x_2/x_1, x_1)$ and $\Phi(x) = (x_1+x_2, x_1)$, respectively. In this case calculations are a lot easier. We can relax this to $\Phi$ being a diffeomorphism ``almost everywhere'', in the following sense:
\begin{assumption} \label{ass:diffeo}
We assume that $m=n$ and there is a set $S\subset \R^n$ with full measure $\mu(S)=1$ such that $\Phi : S\to \Phi(S)\subset \R^n$ is differentiable and invertible such that $\Phi^{-1}: \Phi(S) \to S$ is differentiable as well.
\end{assumption}
\begin{remark}
Examples for such almost everywhere diffeomorphisms (in addition to the guiding example) include polar coordinates on $\R^2$ if $\mu$ is absolutely continuous with respect to the Lebesgue measure: The ``stitches'' of the polar coordinates where the coordinate transform is not invertible is a set of measure $0$.
\end{remark}

In this case, we can explicitly calculate the fan measure by a integral transformation. Take again a measurable set $A = \Phi^{-1}(\{s\}\times B)$, then

\begin{align}
\mu_{\Phi, s}(A) &= \lim_{\eps\to 0}\frac{\frac1 {\vert B_\eps(s)\vert }\int_{B_\eps(s)}\int_{B}f_\mu(\Phi^{-1}(t,u))\cdot \vert \det J\Phi^{-1}(t,u)\vert  \d u \d t}{\frac1 {\vert B_\eps(s)\vert }\int_{B_\eps(s)}\int_{\R^{n-k}}f_\mu(\Phi^{-1}(t,u))\cdot \vert \det J\Phi^{-1}(t,u)\vert  \d u \d t} \notag \\
&=\frac{\int_{B}f_\mu(\Phi^{-1}(s,u))\cdot \vert \det J\Phi^{-1}(s,u)\vert  \d u }{\int_{\R^{n-k}}f_\mu(\Phi^{-1}(s,u))\cdot \vert \det J\Phi^{-1}(s,u)\vert  \d u }. \label{eq:fanmeasure_diffeo}
\end{align}
This means that the fan measure has a ``surface integral density'' which is weighted by the determinant term from the transformation mapping $\Phi$. As different choice of $\Phi$ correspond to different weightings, this immediately shows that the Borel--Kolmogorov paradox must appear.

\paragraph{Example: Nonlinear Shearing}
Consider cartesian coordinates $x=(x_1, x_2,\dots,x_n)\in \R^n$  and define
\begin{align*}
\Phi(x) &= \begin{pmatrix}
\varphi(x)\\ x_2 \\ \vdots \\ x_n
\end{pmatrix}
\end{align*}
This is, for example, the setting of the guiding example and the choices of parametrizations considered. 
Assume that assumption \ref{ass:diffeo} holds. 
The inverse mapping is of the form $\Phi^{-1}(y) = (\chi(y), y_2, \ldots, y_n)$ and 
\[J\Phi^{-1}(y) = \begin{bmatrix}
\frac{\partial \chi}{\partial y_1}(y)& \frac{\partial \chi}{\partial y_2}(y)  & \frac{\partial \chi}{\partial y_3}(y) & \cdots & \frac{\partial \chi}{\partial y_n}\\
0 & 1 & 0 & \cdots & 0\\
0 & 0 & 1 & \cdots  & 0\\
\vdots & \vdots & \vdots & \ddots & \vdots\\
0 & 0 & 0 & \cdots & 1
\end{bmatrix}\]
and thus $\det J\Phi^{-1}(y) = \frac{\partial \chi}{\partial y_1}(y)$. We can obtain information about $\chi$ by applying the implicit function theorem: In order to have $\Phi(\Phi^{-1}(y)) = y$, we need to have $y_1 = \varphi(\chi(y), y_2, \ldots, y_n)$ for all $y\in \R^n$. 

In order to now leverage the implicit function theorem, we look at the function $f(y, w) = \varphi(w, y_2, \ldots, y_n) - y_1$ and we are interested in the case $f(y, w) = 0$. 

Consider any point $y$ in the set $\Phi(M)$ where $\Phi^{-1}$ is a diffeomorphism. In particular, $\frac{\partial \varphi}{\partial x_1}(\Phi^{-1}(y))\neq 0$. The implicit function theorem yields the existence of a mapping $\chi$ such that $f(y, \chi(y)) = 0$ in a neighborhood of $y$ and that 
\begin{align*}
\frac{\partial \chi}{\partial y_1}(y) &= \left(-\frac{\partial f}{\partial w}(y, \chi(y))\right)^{-1}\cdot \left(\frac{\partial f}{\partial y_1}(y, \chi(y))\right) =  \left(\frac{\partial \varphi}{\partial x_1}( \chi(y), y_2,\ldots,y_n)\right)^{-1} \\
&= \left(\frac{\partial \varphi}{\partial x_1}(\Phi^{-1}(y))\right)^{-1}
\end{align*}

Hence, 
\[\vert \det J\Phi^{-1}(y)\vert  = \left\vert \frac{\partial \varphi}{\partial x_1}(\Phi^{-1}(y))\right\vert ^{-1}.\]

Returning to our characterization \eqref{eq:fanmeasure_diffeo} of the fan measure:

\begin{align*}
\mu_{\Phi, s}(A) &=\frac{\int_{B}f_\mu(\Phi^{-1}(s,u))\cdot \vert \det J\Phi^{-1}(s,u)\vert  \d u }{\int_{\R^{n-1}}f_\mu(\Phi^{-1}(s,u))\cdot \vert \det J\Phi^{-1}(s,u)\vert  \d u }\\
&=\frac{\int_{B}f_\mu(\Phi^{-1}(s,u))\cdot \left\vert \frac{\partial \varphi}{\partial x_1}(\Phi^{-1}(s,u))\right\vert ^{-1} \d u }{\int_{\R^{n-1}}f_\mu(\Phi^{-1}(s,u))\cdot \left\vert \frac{\partial \varphi}{\partial x_1}(\Phi^{-1}(s,u))\right\vert ^{-1} \d u }
\end{align*}

We can see that we could have resolved the apparent paradox in the guiding example by this approach as well: The two different parametrizations correspond to a choice of $\varphi(x_1,x_2) = x_2/x_1$ and $\varphi(x_1,x_2) = x_1+x_2$, respectively. The first option directly yields the additional factor of $\vert u\vert $ in the resulting density, whereas the second one coincides with the canonically induced submeasure since $D\varphi$ is constant.

\paragraph{Example: Polar coordinates in 2d}
We consider cartesian coordinates $(x_1, x_2) \in \R^2$ and the coordinate transformation $x_1 = y_1 \cos(y_2), x_2 = y_2\sin(y_2)$ where we (differing from convention) set $y_1\in \R$ (instead of $\R^+$) and $y_2\in (0, \pi)$. The mappings are 

\begin{align*}
\Phi(x) &= \begin{pmatrix}\sign(x_2)\cdot \sqrt{x_1^2+x_2^2} \\
\arccos \left(\frac{x_1}{\sqrt{x_1^2+x_2^2}}\right)
\end{pmatrix} = \begin{pmatrix}
\varphi(x)\\ \psi(x)
\end{pmatrix}\\
\Phi^{-1}(y) &= \begin{pmatrix}
y_1 \cos y_2\\ y_1\sin y_2
\end{pmatrix}
\end{align*}
with $\vert \det (J\Phi^{-1}(y))\vert  = \vert y_1\vert $. If $X$ is a random variable on $\R^2$ with density $f$, then $Y = \Phi(X)$ has density $g(y) = f(y_1\cos y_2, y_1\sin y_2)\cdot \vert y_1\vert $.

\begin{figure}[hbtp]
\centering
\includegraphics[width=0.5\textwidth]{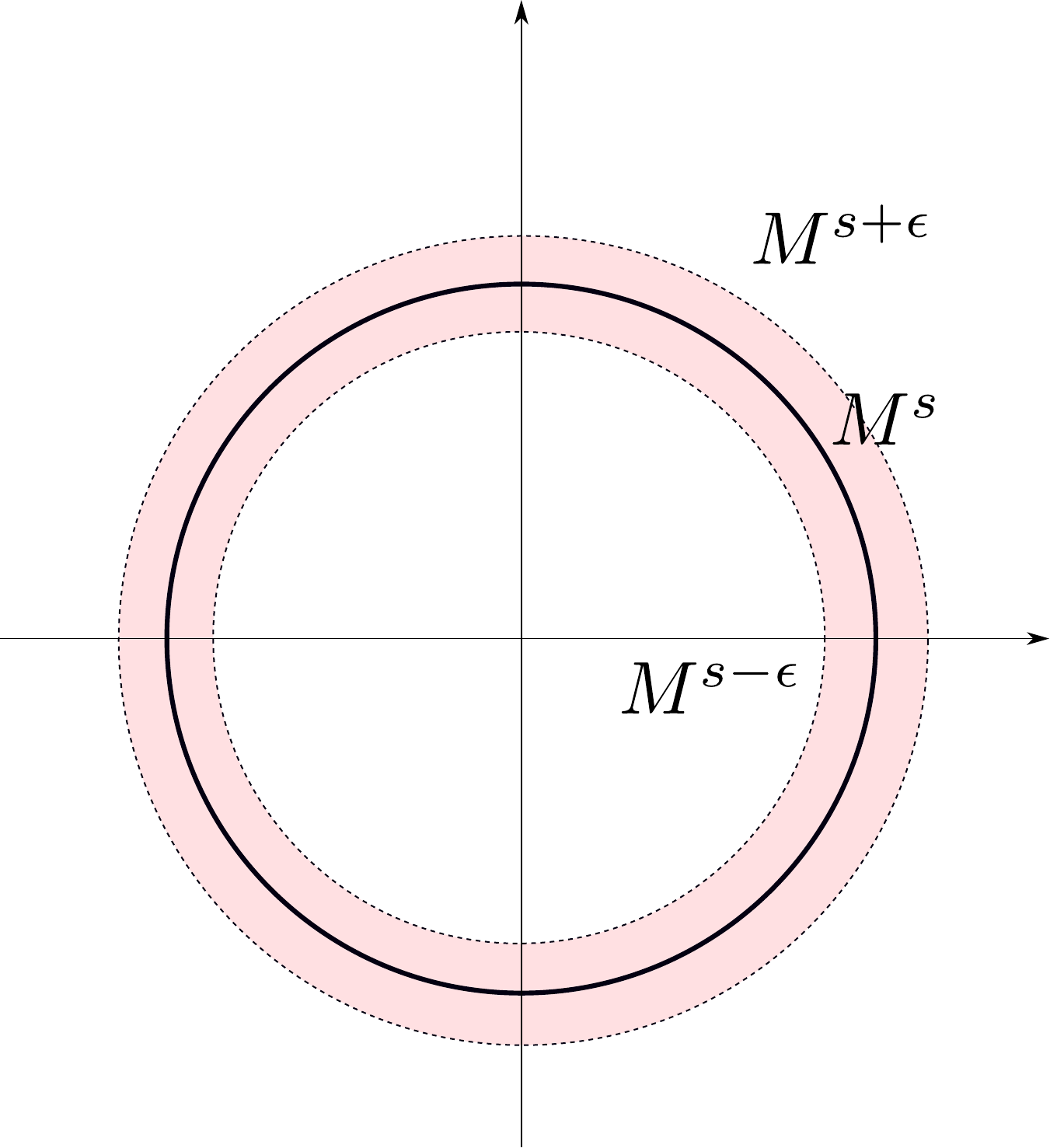}
\caption{$M^s$ is uniformly approximated by rings of vanishing thickness.}
\end{figure}

We first consider the sphere of radius $s$ in two dimensions $M^s := \{x\in \R^2 \st \vert x\vert  = s\} = \{x \st \varphi(x) = s\}$.  It is easily seen that $\vert D\varphi(x)\vert  = 1$ for $x \in M^s$ (this is due to the fact that concentrical spheres have constant distance to each other everywhere), hence theorem \ref{thm:main} states that the canonically induced measure $\mu_{M^s}$ coincides with the chart dependent measure $\mu_{\Phi, s}$. This in turn (according to lemma \ref{lem:conditioning_on_rv}) is the same as the conditional distribution of $W = \psi(X)$ (the polar angle) given by $\varphi(X) = s$ (the radius), and is given by
\[\mu_{\Phi,s}(\Phi^{-1}(\{s\}\times B))= \frac{1}{C_s}\int_B f(s\cdot \cos u, s\sin u)\d u,\]
with a suitable constant $C_s>0$ (cf.~\eqref{eq:fanmeasure_diffeo}).
If, for example, $X\sim N(0,1)\otimes N(0,1)$, then this is the uniform measure on the sphere with radius $s$.

\begin{figure}[hbtp]
\centering
\includegraphics[width=0.5\textwidth]{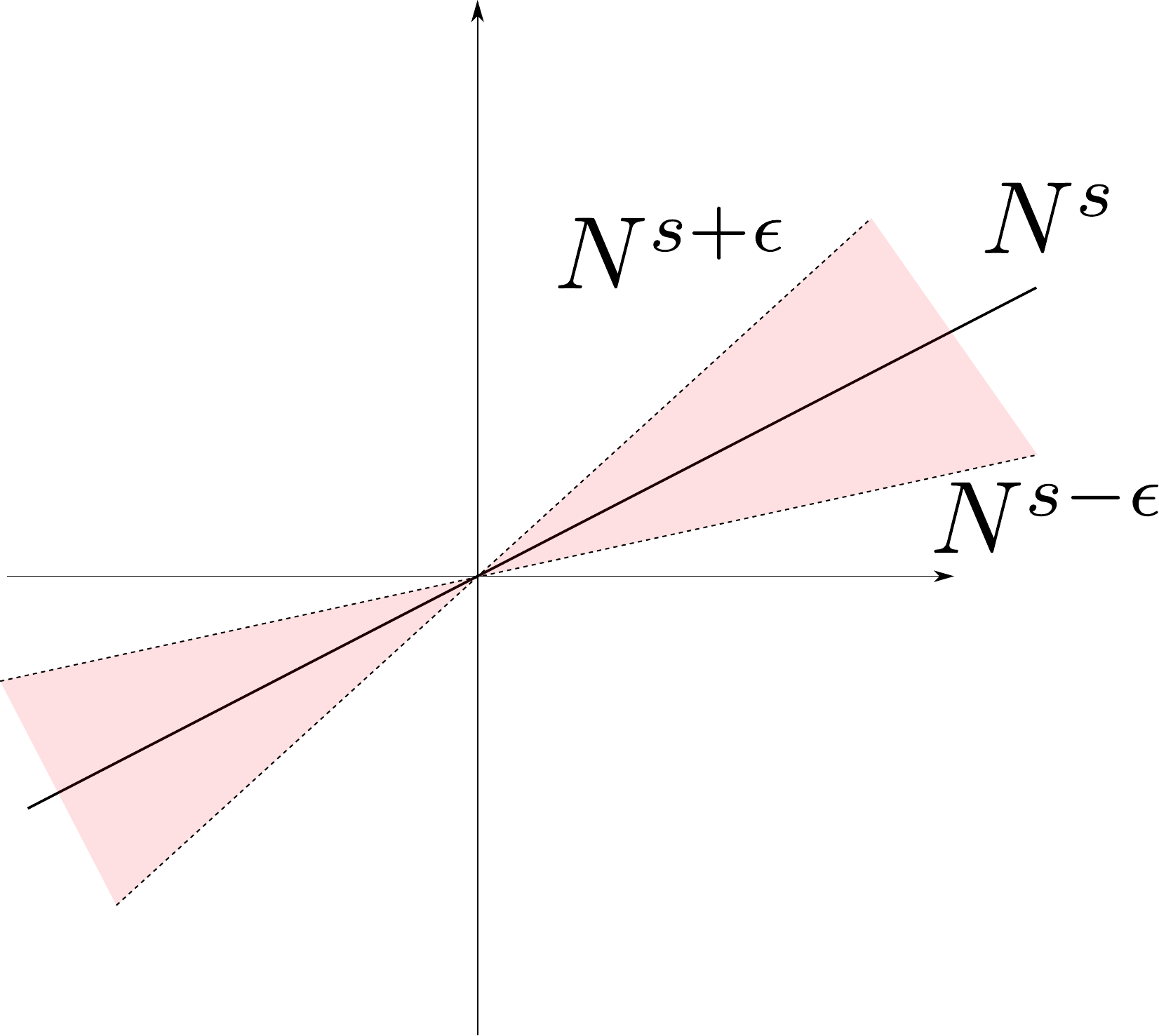}
\caption{$N^s$ is (non-uniformly) approximated by cones of vanishing angle. This means that points far away from the origin are weighted more strongly because of the wide opening of the approximating cone. This is the significance of the factor $\vert u\vert $ in the conditional measure.}
\end{figure}

If, on the other hand, we are interested in the subset $N^s := \{x\in \R^2: x_2 = s\cdot x_1\}$ (a straight line through the origin), then things turn out differently: For convenience we switch the order of the coordinates, i.e., 
\begin{align*}
\Phi(x) &= \begin{pmatrix}
\arccos \left(\frac{x_1}{\sqrt{x_1^2+x_2^2}}\right)\\
\sign(x_2)\cdot \sqrt{x_1^2+x_2^2} 
\end{pmatrix}
\end{align*}
such that $N^s = \{x\in \R^2: \varphi(x) = \arccos \frac{\operatorname{sign}(x_1)}{\sqrt{1+s^2}}\}$. Then 
 $D\varphi$ does not have constant norm on $N^s$ and the conditional measure induced by $\varphi = s$ is not the canonical measure: In particular,
 \[\mu_{\Phi,s}(\Phi^{-1}(\{s\}\times B)) = \frac{1}{C_s}\int_B f(u \cos s, u\sin s)\vert u\vert  \d u.\]

\section{Practical considerations on dealing with singular conditional probabilities}\label{sec:solution}
The initial promise of this manuscript was to give the reader a clearer idea of how to handle or avoid being bit by the Borel--Kolmogorov paradox. From the discussion so far we propose the following. 

\subsection{Conditioning on a set versus conditioning on a random variable}
Let's assume that we want to condition on a singular set. The first question which readers have to answer for themselves is, ``Do I want to condition on a singular set or on a singular event given by a random variable?''.

For the example of section \ref{sec:guidingexample}, this means finding out whether we are interested in
\begin{enumerate}
    \item the distribution of $X$ on the singular set $\{(x_1,x_2) \st x_1+x_2=0\}$ or
    \item the distribution of $X$ given $V=c$ from some random variable $V$ and $c\in \R$ (with, e.g., $V=W$ or $V=Z$).
\end{enumerate}

In the first case, we propose to use the canonically induced measure. This can either be computed by definition (but the Hausdorff measure is notoriously difficult to handle explicitly) or by choosing a specific parametrization via a random variable for which the conditions of theorem \ref{thm:main} hold. Then the fan measure (i.e., the ``vanilla conditional distribution'' due to lemma \ref{lem:conditioning_on_rv}) can be computed and will correspond to the canonically induced measure.

In the second case, one should ignore the canonically induced measure and stick with the parametrization induced by the random variable $V$ (a more thorough motivation for that will be given in the next section).

\subsection{The relevance of the Borel--Kolmogorov paradox for Bayesian inverse problems}
The Borel--Kolmogorov paradox only arises when we condition on \textit{subsets} instead of random variables. This is also called the \textit{equivalent events fallacy}, because it makes a big difference what random variables we use to represent a given event.

On one hand, the Borel--Kolmogorov paradox is not directly related to Bayesian inverse problems: Bayes' theorem instructs us to always condition on the random variable which is measured without any geometric subset involved.
Still, the similarity between these two settings is too large to ignore and it pays off to be cautious.

We consider again the example from section~\ref{sec:guidingexample} with cartesian coordinates $(x_1, x_2)\in \R^2$ and a random variable $X = (X_1, X_2)\in \R^2$ in this space. 
We are interested in conditioning the random variable $X$ on the set $M:=\{(x_1, x_2) \st x_1 + x_2 = 0\} = \{(x_1, x_2) \st x_2/x_1 = -1 \}$. 
These two representations are of course equivalent, but lead to different conditional measures on $M$, as we saw in section~\ref{sec:guidingexample}.

The canonically induced measure, as it is defined in definition \ref{def:canonical_submeasure}, it is not straightforwardly computed, so in practice we condition on a different random variable. 
In section~\ref{sec:guidingexample} we compared two different conditioning procedures: 
We defined $(Z, \tilde X):= \Phi_1(X_1, X_2) := ({X_2}/{X_1}, X_1)$, and $(W, \tilde X) := \Phi_2(X_1, X_2) := (X_1+X_2, X_1)$ and identified $M$ with the singular events $\{Z = -1\}$ and $\{W = 0\}$. 
The paradox arises as soon as we use those random variables for an implicit transformation of the probability space.
Lemma~\ref{lem:conditioning_on_rv} shows that the distribution of $\tilde X = X_1$ given $Z=-1$ is the measure $\mu_{\Phi_1, -1}$ and the distribution of $\tilde X = X_1$ given $W = 0$ is the measure $\mu_{\Phi_2, 0}$.
It is easily calculated that $\vert \det (J\Phi_1^{-1}(y))\vert  = \vert y_2\vert $ and $\det (J\Phi_2^{-1}(y))\vert  = 1$. 
Hence, theorem~\ref{thm:main} implies that $\mu_{\Phi_2, 0}$ is the canonically induced measure, whereas $\mu_{\Phi_1, -1}$ is not.

Graphically, we can also make sense of this statement (see Figure~\ref{fig:N_and_M}). If we define $M^s := \{x\in \R^2: \Phi_1(x) = s\}$ and $N^s := \{x\in \R^2: \Phi_2(x) = s\}$, then $M = M^{-1} = N^0$.
The characterization $M^s$ implies approximation by cones of vanishing angle (and thus a non-uniform approximation). This is due to the fact that cones do not have constant distance from each other. The characterization $N^s$ approximates by ``concentric'' tubes of vanishing width and thus approximates uniformly.
\begin{figure}[hbtp]
\centering
\includegraphics[width=0.45\textwidth]{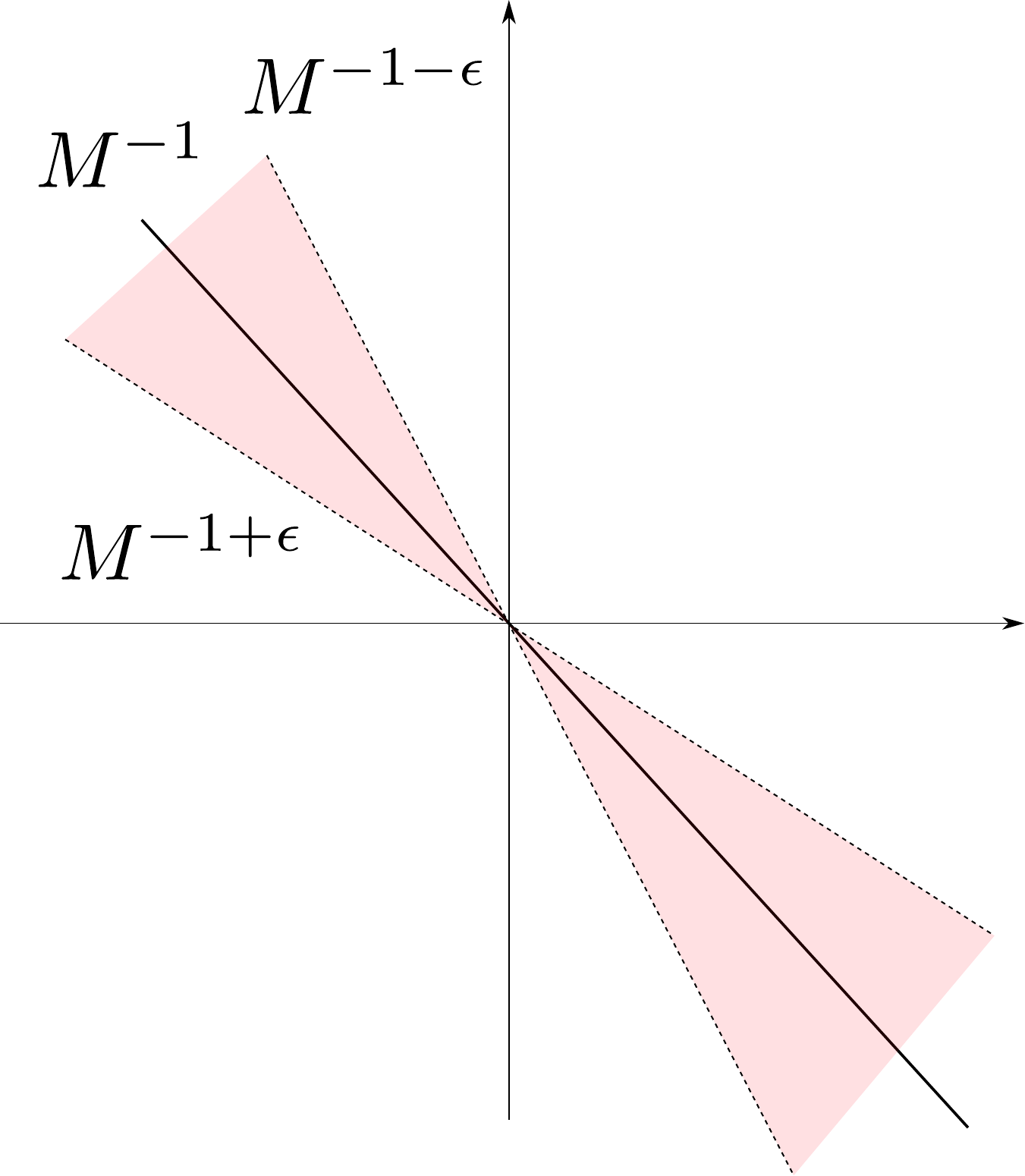}\hfill
\includegraphics[width=0.45\textwidth]{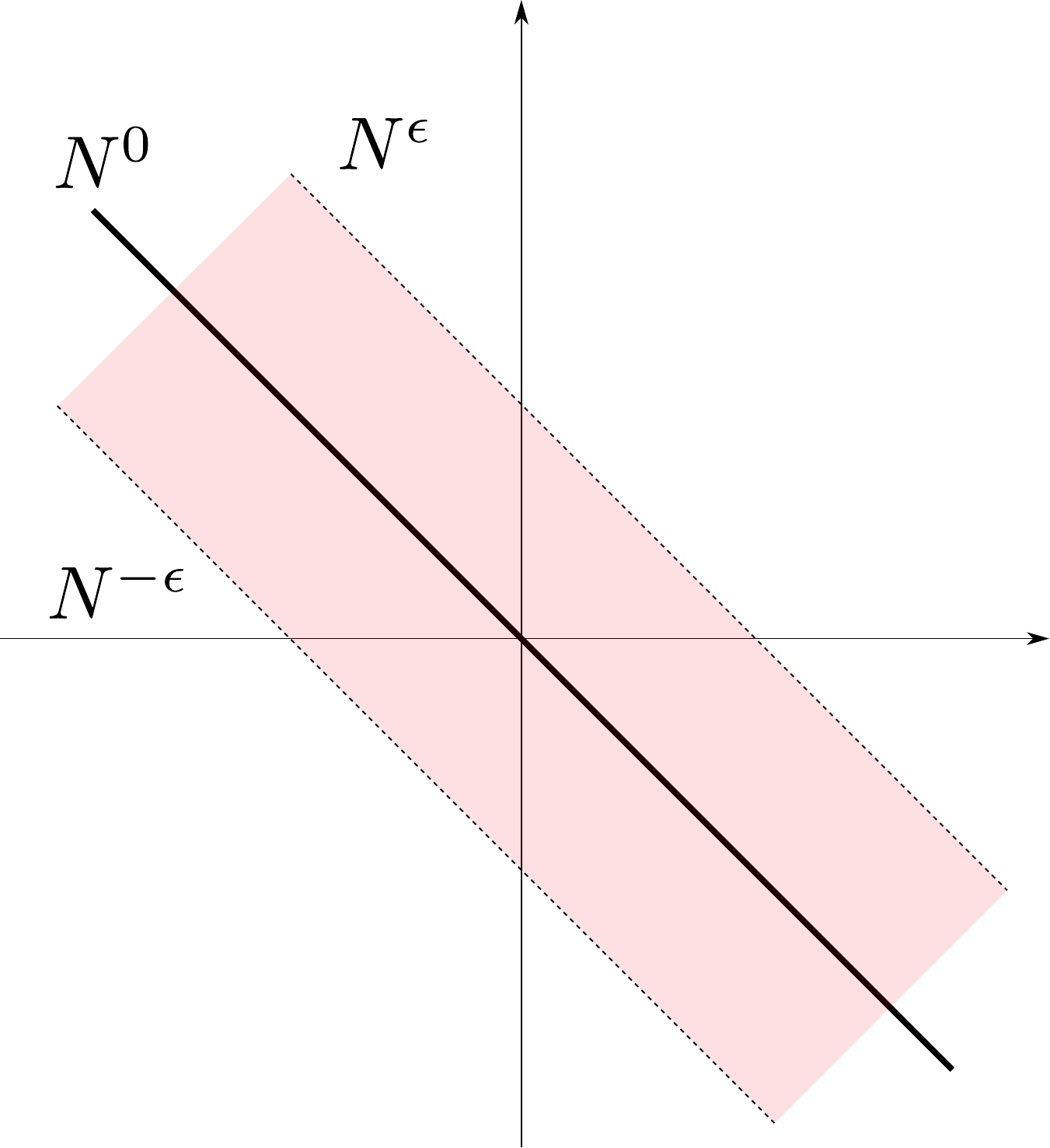}
\caption{Non-uniform and uniform approximation of the set $M=\{x_1+x_2=0\}$}
\label{fig:N_and_M}
\end{figure}

This seems to suggest that even if the measured data is given by 
\[ Z = \frac{X_2}{X_1} = -1,\]
one needs to forget the physical measurement $Z$ and instead condition on the entirely synthetic event $W = X_1 + X_2 = 0$ in order to obtain a ``canonical conditional probability'' on the event $M$. 
This seems strange or even wrong and one may wonder how canonical $\mu_M$ really is. 

We will see that this apparent contradiction arises from the definition of the probability space and can be corrected.
To see this reformulate the two conditioning procedures discussed above as follows:
\begin{enumerate}[A)]
\item Let $X = (X_1,X_2)$ be a $N(0,1)\otimes N(0,1)$ random variable. What is the distribution of $X_1$ given that $X_1 = -X_2$?
\item Let $X_1 \sim N(0,1)$ be a parameter and $X_2 \sim N(0,1)$ be a noise vector. What is the distribution of $X_1$ given we measure $Z = \frac{X_2}{X_1}$ to be $Z=-1$?
\end{enumerate}

Formulation~A is clearly a situation where we can apply our notion of canonically induced submeasure (because we condition on a set) and from our calculations above we derive that the correct conditional distribution of $X_1$ is given by $X_1\vert (W=0)$ with $W = X_1 + X_2$.

Formulation~B differs only in the interpretation of the quantities involved: $X_1, X_2$ and $W$ are given clearer meaning. 
In particular, we do not condition on a geometric set (like $\{(x_1, x_2): x_1 = -x_2\}$) but on a measurement event, which makes all the difference.
The correct notion of conditional probability in this case \textit{is} the version conditioned on $Z = -1$, although our preceding calculation implies that this does not agree with the canonically induced submeasure. 

The issue here is the following: 
The example above starts with the probability space $(\R^2, \B(\R^2), \mu=\mu_0\otimes \gamma)$ on the coordinates $X_1$ and $X_2$.
Here $\mu_0$ is the prior measure on $X_1$ and $\gamma$ is the ``noise'' measure on $X_2$, if we interpret $X_1$ as a parameter and $X_2$ as a noise term. 
In particular, the two variables are independent. 

In formulation A, this constitutes the canonical choice for the fundamental probability space: $Z$ (and $W$, if it is defined) are only artificial quantities, but $X_1$ and $X_2$ are the fundamental variables with respect to which $\mu$ is defined.

In formulation B, the Bayesian inverse problem, we argue that the correct tuple for constructing the fundamental probability space is $(X_1, Z)$:
Even without noise a measurement process is never exact, e.g., because measurement instruments use binning or can only be read of with limited accuracy.

In other words, a measurement in the context of a Bayesian inverse problem is never actually $Z = -1$, but rather $Z \in (-1-\varepsilon,-1+\varepsilon)$ with $\varepsilon$ a read-off (or digitalisation) uncertainty. 
Hence, conditioning $X_1$ on the singular case $Z=-1$ should arise as a limit of conditioning $X_1$ on the non-singular sets $(-1-\eps,-1+\eps)$: 
From the proof of lemma \ref{lem:conditioning_on_rv} we see that this corresponds to the usual conditional probability distribution $X_1\vert (Z=-1)$. 
Starting with the probability space $(\R^2, \B(\R^2), \mu=\mu_0\otimes \gamma)$, we obtain $Z$ only via the transformation $Z = \varphi(X_1, X_2) = {X_2}/{X_1}$ which leads to the apparent contradiction between canonical conditional probability and ``what we actually want''. 

In contrast to the Borel--Kolmogorov paradox, where there are no distinguished coordinates, in an Bayesian inverse problem we have more structure: 
While there are two degrees of freedom, the parameter $X_1$ and the noise $X_2$, together with the data $Z$ they constitute three dependent variables and a probability space can be defined using any two of them. 
When we choose $(X_1, Z)$ as the elementary coordinates of the probability space, and the probability measure on those variables as push-forward of $\mu$ via the mapping $\Phi$ (in the sense of section \ref{sec:resolutionBK}), then the apparent contradiction vanishes: The customary conditional probability (i.e., plain conditioning to the random variable $Z$) coincides with the canonical conditional probability distribution, because there is no additional transformation map (we condition on the second coordinate).

In order to obtain a non-overdetermined description, we have to choose a pair of random variables out of the triple $(X_1, X_2, Z)$ and use this coordinate description to define the set we want to condition on, i.e., $M = \{Z=z\}$. If we choose $(X_1, X_2)$, we get a measure which is independent in its coordinates (because prior and noise are usually assumed to be independent), but then the canonical measure on $M$ does not correspond to the classical conditional measure given $Z=z$. On the other hand, if we choose $(X_1, Z)$, then the canonical measure on $M$ is identical to the classical conditional measure.

As a quick synopsis of this section we give the following proposition that describes how to reconcile Bayesian inverse problems with the Borel--Kolmogorov paradox.
\begin{proposition}[Bayesian inverse problems and the Borel--Kolmogorov paradox]\label{lem:BIP_BK}
Consider a Bayesian inverse problem of recovering an unknown vector $\mathcal X_1 \ni X_1\sim \mu_{X_1}$ (where $\mu_{X_1}$ is the prior measure on $X_1$, i.e., the law of $X_1$)  from a known measurement 
\[ \mathcal Z \ni Z = G(X_1, X_2)\]
with independent noise $\mathcal X_2 \ni X_2\sim \mu_{X_2}$. The product measure is defined as $\nu := \mu_{X_1}\otimes \mu_{X_2}$. Then $\nu$ and $G$ define a probability space $\mathcal Y = (\mathcal X_1\times \mathcal Z, \sigma(\mathcal X_1\times \mathcal Z), \mu)$ where $\mu$ is the push-forward measure $\mu(C) = \hat G_\#\nu(C)$ with $\hat G(x_1,x_2) = (x_1, G(x_1, x_2))$.

For concreteness, we say that the known measurement of $Z$ is given by $s$. We define $M := \{(x_1, z)\in \mathcal Y: z = s\}$.

In this probability space $\mathcal Y$, the posterior probability distribution of $X_1\vert {Z=s}$ with density $f_{X_1\vert Z=s}(x_1) := \frac{f_{X_1,Z}(x_1,s)}{f_Z(s)}$ coincides with the canonically induced measure $\mu_M$.
\begin{proof}
Define $\varphi(x_1,z) = z$ and $\psi(x_1,z) = x_1$, then the set $M$ is also described by $\varphi^{-1}(\{s\})$. 
The fan measure $\mu_{\Phi, s}$ is plain conditioning to the second component and thus corresponds to the usual conditional measure with density $f_{X_1\vert Z=s}(x_1)$ according to lemma \ref{lem:conditioning_on_rv}. On the other hand, $D\varphi \equiv 1$ on $M$, and thus by theorem \ref{thm:main}, it is identical to the canonical measure $\mu_M$.
\end{proof}
\end{proposition}

\begin{appendix}
\section*{Appendix: A discussion of \cite{gyenis2017conditioning}}
In this section we discuss an approach by Gyenis, Hofer-Szabo and Redei, and we demonstrate with a counterexample that it cannot be used to circumvent the Borel--Kolmogorov paradox. The idea of \cite{gyenis2017conditioning} is to reduce the question of singular conditioning to the following concept of statistical inference:

Let $(X,\mathcal S, p)$ be a probability space, $\calA$ a sub-$\sigma$-algebra of $\mathcal S$. Assume that $\psi_A$ is a $\|\cdot\|_1$-continuous linear functional on $\mathcal L^1(X,\calA, p_{\calA})$ determined by a probability measure $q_{\calA}$ given on $\cal A$ via
\[\psi_{\cal A}(f) = \int f\d q_{\calA}.\]
What is the extension $\psi$ of $\psi_{\calA}$ from  $\mathcal L^1(X,\calA, p_{\calA})$ to a $\|\cdot\|_1$-continuous linear functional on  $\mathcal L^1(X,\mathcal S, p)$?

The answer proposed by \cite{gyenis2017conditioning} is to define
\[\psi(f) := \psi_{\calA}(\calE(f|\calA)),\]
where $\calE(f|\calA)$ is the conditional expectation of $f$ with respect to the $\sigma$-algebra $\calA$. Then the measure $q(B):=\psi(\chi_B)$ is an extension of $q_{\calA}$ on the whole probability space $(X,\mathcal S, p)$.

Now the question of singular conditioning boils down to choosing $\calA = \{\emptyset, A, A^c, X\}$, setting $q_{\calA}(A) = 1$, $q_{\calA}(A^c) = 0$, and defining the extension $q$ as the $A$-conditional probability measure on $X$.

We show that this unfortunately does not work if $A$ is an atom with respect to the probability measure $q$, by way of a concrete example in one dimension:

Our underlying probability space is $([0,1], \mathcal B([0,1]), \lambda)$ with $\lambda$ the Lebesgue measure. We first condition wrt to $\A := \{\emptyset, \{0\}, (0,1], [0,1]\}$. 

\paragraph{The conditional expectation.}
By the measurability requirement, any version of conditional expectation is of the form
$ \calE(f\vert \A) = \chi_{\{0\}}\cdot C_1 + \chi_{(0,1]}\cdot C_2.$ The integration criterion yields $C_2 = \int_0^1 f(x)\d x$ but cannot resolve the value of $C_1$ (because $\{0\}$ is a 0-set of $\lambda$). This means that 
\[\calE(f\vert \A) = C\cdot \chi_{\{0\}} +  \int_0^1 f(x)\d x  \cdot \chi_{(0,1]}\]
is a valid choice for conditional expectation for any $C\in  \R$.

\paragraph{Choosing a measure $q_\A$ on $\A$.}
We set $q_\A(\{0\}) = \rho$ and $q_\A((0,1]) = 1-\rho$ for arbitrary but fixed $\rho \in [0,1]$. Now assume a $\A$-measurable function $g$, necessarily of the form $g = C_1\cdot \chi_{\{0\}} +  C_2 \cdot \chi_{(0,1]}$. Then the linear functional $\psi_\A$ applied to this is defined as 
\[\psi_\A(g) = \rho\cdot C_1 + (1-\rho)\cdot C_2.\]

\paragraph{Extending the measure $q_\A$ to a measure $q$ on $\mathcal B([0,1])$.} We define the extension of the linear functional by $\psi(f) = \psi_\A(\calE(f\vert \A)$ for any $f$ which is Borel-measurable. This is
\[\psi(f) = \rho\cdot C + (1-\rho)\cdot \int_0^1 f(x)\d x.\]
The extension of the measure is then defined as $q(B) = \psi(\chi_B)$.

\paragraph{When is $\psi$ a proper extension of $\psi_\A$?}
We need to derive conditions such that $q(\{0\}) = q_\A(\{0\})$ and $q((0,1]) = q_\A((0,1])$. Concerning the first set, write $A = \{0\}$. Then 
\[q(A) = \psi(\chi_A) = \psi_\A(\calE(\chi_A\vert \A)) = \rho\cdot C + (1-\rho) \cdot 0 = \rho\cdot C.\]
If this is supposed to be equal to $q_\A(A) = \rho$, we need $C = 1$.

Secondly, write $A^c = (0,1]$. Then (using $C=1$)
\[q(A^c) = \psi(\chi_{A^c}) = \rho\cdot C + (1-\rho) = 1.\]
But this is only equal to $q_\A(A^c) = 1-\rho$, if $\rho = 0$, i.e. if $\{0\}$ has $q_\A$-measure $0$ and $q_\A$ is thus necessarily absolutely continuous with respect to $p = \lambda$. This is in contradiction to our assumption that $q_{\A}(A) = 1$ and $q_\A(A^C) = 0$.

\paragraph{Comparison to remark 4 in the paper} The paper states in remark 4 that $\psi$ will be an extension of $\psi_\A$ if $p(A) = 0$ (this is the case here) and $q_\A(A) = 1$. But if we enforce this condition (which means that $\rho = 1$), then from above we know that $q(A^c)\neq q_\A(A^c)$ and thus $q$ is not an extension of $q_\A$.

The reason why things break down here is the following: On the one hand, from above,
\begin{displaymath}
\calE(\chi_{(0,1]}\vert \A) = \chi_{(0,1]} + C \cdot \chi_{\{0\}}.
\end{displaymath}
On the other hand, $\chi_{(0,1]}$ is $\A$-measurable, i.e. is unchanged by conditional expectation and 
\[ \calE(\chi_{(0,1]}\vert \A) = \chi_{(0,1]}.\]
This makes sense, because conditional expectation is only define up to $p$-zero-sets (of which $\{0\}$ is one. But this ambiguity now makes a huge difference because $q_\A$ poses a non-zero probability here: 
\[q((0,1]) = \psi(\chi_{(0,1]}) = C\rho + 1-\rho\]
and as every value of $C$ is equally valid, there is no canonical extension of $q_\A$. Even if we single out a version by setting $C = 0$ (i.e. such that $q((0,1]) = q_\A((0,1]) = 1-\rho$), then the other set makes problems:
\[q(\{0\}) = \psi(\chi_{\{0\}}) = (1-\rho)\cdot \int_0^1 \chi_{\{0\}}(x)\d x = 0\]
which is a violation of the requirement $q(\{0\}) = q_\A(\{0\})$.

If we set $C=1$, then $q(\{0\}) = q_\A(\{0\})$ but $1 = q((0,1]) \neq q_\A((0,1]) = 1-\rho$. Hence we can never choose a consistent version of conditional expectation (which is not pointwise defined anyway) such that we can extend $q_\A$ to $q$. This is because the conditional expectation has an arbitrary value on a set  ultimately due to the fact that $q_\A$ is not absolutely continuous wrt $p$.

We can generalize this example:
\begin{lemma}
Let $(X, \S, p)$ be a probability space and $\A$ be a sub-$\sigma$-field of $\S$ and $q_\A$ a probability measure on $(X, \A)$. Define the $\Vert \cdot\Vert _1$-continuous linear functional $\psi_\A$ defined via $q_\A$, i.e.
\[ \psi_\A(g) = \int_X g \mathrm d q_\A.\]
Assume that $q_\A$ is not absolutely continuous wrt $p$. Then there is no consistent extension of $q_\A$ to all of $q$ via
\[q(B) = \psi(\chi_B) ~\dot =~ \psi_\A (\calE(\chi_B\vert \A)). \]
\end{lemma}
\begin{proof}
Consider a set $A\in \A$ such that $q_\A(A) > 0$ but $p(A) = p_\A(A) = 0$. Then for consistency we need $q(A) = q_\A(A)$, hence we compute
\begin{align*}
q(A) &= \psi(\chi_A) = \psi_\A (\calE(\chi_A\vert \A)) \\
\intertext{Now $\chi_A$ is $\A$-measurable, and thus $\calE(\chi_A\vert \A) = \chi_A$. But $\calE(\cdot\vert \A)$ is only defined up to sets of $p$-measure $0$, hence $\calE(\chi_A\vert \A) = C\cdot \chi_A$ are valid versions for all values of $C$}
	&= \psi_\A(C\cdot \chi_A) = C\cdot q_\A(A) \stackrel{!}= q(A)
\end{align*}
(where the last equality is necessary for consistency of the two measures), and thus we need to choose $C=1$.

On the other hand, 
\begin{align*}
q(A^c) &= \psi(\chi_{A^c})= \psi_\A (\calE(\chi_{A^c}\vert \A)) \\
\intertext{(again, $A^c$ is $\A$-measurable but we need to account for arbitrariness in $A$)}
	&= \psi_{\A}(C\cdot \chi_A + \chi_{A^c}) = C\cdot q_\A(A) + q_\A(A^c) \stackrel ! = q(A^c)
\end{align*}
and thus we need to choose $C=0$. Hence even if we could ``nail down'' the conditional expectation on the set $A$ by setting the constant $C$ (which we cannot), there is no consistent way of doing so.
\end{proof}

\end{appendix}

\printbibliography

\end{document}